\documentclass[11pt]{article}
\usepackage{amsmath, amssymb, amscd, amsthm, amsfonts}
\usepackage{graphicx}
\usepackage{hyperref}
\usepackage[margin=1in]{geometry}
\usepackage{color}
\usepackage{caption}

\newtheorem{theorem}{Theorem}
\newtheorem{lemma}{Lemma}
\newtheorem{corollary}{Corollary}

\newcommand{\floor}[1]{\lfloor #1 \rfloor}

\newcommand{\wythoff}[0]{Wythoff's Nim}

\title{Wythoff's Nim with Finite Alterations}
\author{Mirabel Hu \and Daniel Sleator \and William Tsin}

\begin{document}
\maketitle

\begin{abstract}
Wythoff's Nim is a variant of 2-pile Nim in which players are allowed
to take any positive number of stones from pile 1, or any positive number of
stones from pile 2, or the same positive number from both piles.  The
player who makes the last move wins.  It is well-known that the
P-positions (losing positions) are precisely those where the two piles
have sizes $\{\floor{\phi n}, \floor{\phi^2n}\}$ for some integer
$n\geq 0$, and $\phi = (1+\sqrt{5})/2 = 1.6180\cdots$.

In this paper we consider an altered form of Wythoff's Nim where an
arbitrary finite set of positions are designated to be P or N
positions.  The values of the remaining positions are computed in the
normal fashion for the game.  We prove that the set of P-positions of
the altered game closely resembles that of a translated normal Wythoff
game.  In fact the fraction of overlap of the sets of P-positions of
these two games approaches $1$ as the pile sizes being considered go
to infinity.

\end{abstract}


\section{Introduction}

Wythoff's Nim~\cite{wythoff1907modification} 
is a variant of two-pile Nim.  The state of the game is
defined by a pair of non-negative integers $(x,y)$, the number of
stones in each of the piles.  Players alternate moving, and a move
consists of removing a positive number of stones from one of the
piles, or alternatively (if both piles are non-empty) removing the
same positive number of stones from both piles.  The player who makes
the last move wins.

A state is called an N-position if the next player has a winning
strategy.  And it is called a P-position if the previous player has a
winning strategy.

\begin{figure}[hbt!]
  \centering
  \includegraphics[width=2.5in]{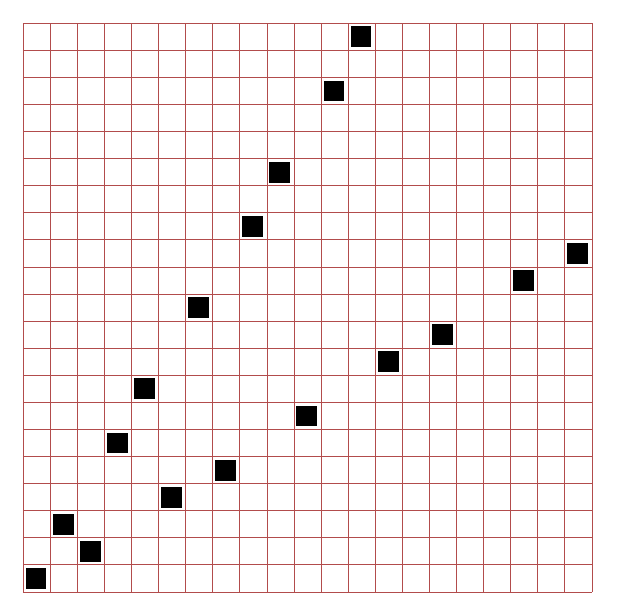}
  \captionsetup{width=.9\linewidth}
  \caption{The P-positions for \wythoff{} with piles up to size 20.  The
    lower left corner is $(0,0)$.}
  \label{fig:wythoff20}
\end{figure}

The P-positions are comprised of two linear beams defined by
the sets $\{(\floor{\phi^2 n}, \floor{\phi n})\} \mid n\geq 0\}$
and $\{(\floor{\phi n}, \floor{\phi^2 n})\} \mid n\geq 0\}$. Here
$\phi = (1+\sqrt{5})/2 = 1.6180\cdots$.  Considering the upper beam,
note that $\floor{\phi^2n} = \floor{(1+\phi) n} = n + \floor{\phi n}$.
It follows that in progressing from one P-position to the next, the
$(dx,dy)$ of each step is either $(1,2)$ or $(2,3)$.  Because of the
irrationality of $\phi$ the pattern of steps sizes is non-periodic.

The following recursive process can be used to label every position
$(x,y)$ for $x,y\geq 0$ as a P or an N position.  Compute the label
for position $(x,y)$ after all positions $(x',y')$ with with $x'\leq
x$, $y'\leq y$, and $(x',y') \neq (x,y)$ have already been computed.
The set of positions reachable from $(x,y)$ in one move are those
obtained by removing one or more from $x$, or removing one or more
from $y$, or removing $z\geq 1$ from both $x$ and $y$.
The label for $(x,y)$ is N if there exists a P-position that is
reachable from $(x,y)$ in one move, otherwise the label of $(x,y)$ is
P.


All of this is well-known.  In fact a cornucopia of variants of
\wythoff{} have been invented and studied~\cite{withoff-visions:2017}.
In this paper we advance the study of \wythoff{} not by changing the rules, but
by considering what happens if we change the initial conditions.  More
specifically a finite list of game states (i.e. the two pile sizes
$(x,y)$) are declared at the outset to be P or N positions.  We'll
call such a game an \textit{altered} \wythoff{} game.  The specific
alterations are represented by two finite disjoint sets $\mathcal{P}$
and $\mathcal{N}$, which indicate the P-positions and N-positions
dictated by the alteration.

Given these alterations, the process of determining the labeling of the
remaining positions is exactly the one described above.  Except that
we take for granted the labels of the altered states, and do not
re-compute them.

For example the mis\`{e}re form of \wythoff{} is the one in which
the player unable to move is the winner.  This is obtained from \wythoff{}
by using alteration sets $(\mathcal{P}, \mathcal{N}) = (\{\},
\{(0,0)\})$.  It's easy to see that this alteration does not
materially change the game.  Specifically the only change is that the
P-positions near the origin (in the box $\{0,1,2\} \times \{0,1,2\}$)
change from $\{(0,0), (1,2), (2,1)\}$ to $\{(0,1), (1,0), (2,2)\}$.
These alternative P-positions in the $3\times 3$ box cover the same
rows (i.e. $x\in\{0,1,2\}$) the same columns (i.e. $y\in\{0,1,2\}$)
and the same diagonals (i.e. $x-y\in\{-1,0,1\}$) as the original set
of P-positions.  Thus the P-positions outside of this box will be
exactly the same as those for \wythoff{}.


What is more surprising is that \textit{any} alteration of
\wythoff{} has \textit{almost} the same set of P-positions as \wythoff{} but
shifted horizontally and/or vertically.  By almost we
mean that as bigger and bigger balls around the origin of the $2$-d
plane are considered, the fraction of points within it that are
P-positions of one of the games but not of the other goes to zero.

\begin{figure}[hbt!]
  \centering
  \includegraphics[width=5.0in]{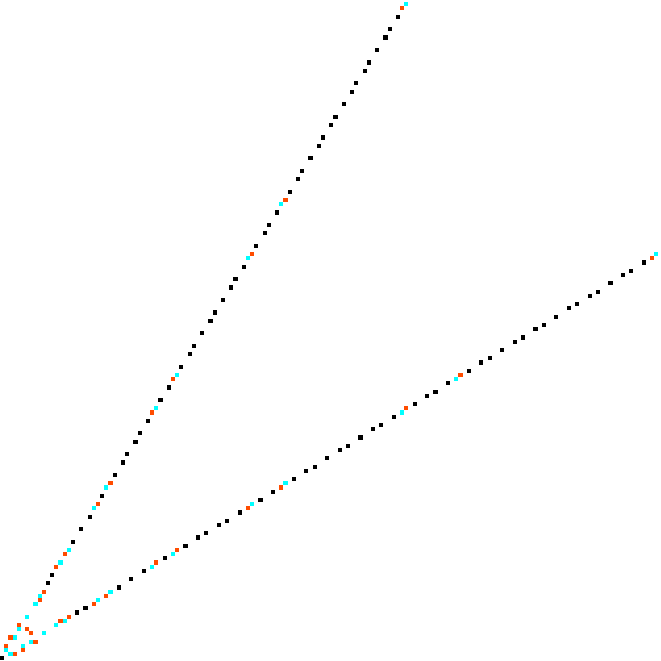}
  \captionsetup{width=.9\linewidth}
  \caption{This diagram shows the difference between \wythoff{} and an
    altered \wythoff{} game.  In the latter game, the lower left $3
    \times 3$ square is all N-positions except for $(0,0)$ which is a
    P-position.  The blue cells are the P-positions of \wythoff{}, the
    red are those of the altered \wythoff{}, and the black cells are
    the P-positions that the two games have in common.  The colored
    ``defects'' which occur in the two beams, will continue to occur
    forever, but will be spread farther and farther apart.  (If you're
    viewing this as a PDF document, you should be able to zoom in to
    see it more clearly.)}
  \label{fig:two-game-diff}
\end{figure}

In section 2 we define the order in which the game is computed, and
present some notation.  Section 3 shows how the crucial part of the
evolution of the the game can be modeled by evolving bitstrings.
Specifically we define the \textit{Wythoff update} which takes a bit
string $S$ and if the first character is a \texttt{1} it is removed
and \texttt{01} is added to the end of the string.  And if the first
character is \texttt{0}, then it is changed to \texttt{1} and a
\texttt{0} is appended to the end of the strting.  In section 4 we
prove a number if useful properties of the Wythoff update.  In section
5 we prove the Unique Offset Theorem which says that the set of
P-positions of any altered game of \wythoff{} is
\textit{asymptotically equal} to the unaltered game with a unique
translation.  Final remarks appear in section 6.

\section{Definitions for Computing and Analyzing altered \wythoff{}}
To prove our results we're going to lay out in detail a specific way
in which the P-positions of the game can be computed.  And we'll show
that as we move into the un-altered region, the game will evolve to
become more and more ordered in specific ways that we will describe.
This eventually will lead us to our main theorem, which says roughly
that any altered game gets closer and closer to the unaltered game.

Let $m_x$ and $m_y$ be such that all the altered game states
$\mathcal{P}$ and $\mathcal{N}$ (described above) are within the box
$[0,m_x-1] \times [0,m_y-1]$.

We will be computing the game's P-positions left to right, column
by column.  By the time we've finished a column, we will have computed
a finite number of P-positions in that column.  All the other
positions in that column are N-positions.  We compute a column $t$ (given
that all the ones to its left have been computed) as follows:
\begin{quote}
  We work our way up, starting from row 0.  For each cell, if it's in
  $\mathcal{P}$ or $\mathcal{N}$ we label it as such, and continue.
  If it's not, then we compute its label as follows: It's P if all
  those in the same column below it are N positions, AND all those in
  the same row to the left of it are N positions, AND all those
  diagonally down and to the left are N positions.  Otherwise it's an
  N position.
\end{quote}
The process terminates when we reach row $m_y$ or above AND at least
one P position has been found in this column.  The process is
guaranteed to terminate because for a column $t$ the maximum height of
a P-position in that column is bounded by $m_y + 2t$.

Let's consider the simpler case of computing a column $t \geq m_x$ (so
all of $\mathcal{P}$ and $\mathcal{N}$ are strictly to the left of
$t$).  A column $t$ is called \textit{natural} if $t \geq m_x$.
We know that the column must contain exactly one P-position.  The
problem is to find it.  We work our way up from row 0, looking for
that P-position as follows:
\begin{quote}
  If the row we're in contains only N-positions, and the diagonal of
  slope 1 down and to the left also contains only N-positions, then
  this point is the P-position in column $t$.
\end{quote}
Say we're testing a row $y$, using the above method.
The only information (about what has been computed so far)
needed to do the test is (1) does row $y$ already contain a
P-position?  And also does the diagonal through $(t-1,y-1), (t-2,
y-2),\ldots$ already contain a P-position?

This information can be captured by two boolean arrays, which we call
\texttt{row$_{t-1}$[]} and \texttt{diag$_{t-1}$[]}.
\texttt{row$_{t-1}$[$y$]} is TRUE if and only
if one of the columns to the left of $t$ contains a P-position in
row $y$.  And \texttt{diag$_{t-1}$[$y-1$]} is TRUE if and only if the
diagonal of slope 1 ending at $(t-1,y-1)$) contains a P-position.
Once the P-position in column $t$ has been computed, the two arrays
can be updated to reflect this change going forward.
(\texttt{diag$_{t}$[]} is obtainied by shifting
\texttt{diag$_{t-1}$[]} up by 1, then editing it for the result of the
P-position found in column $t$.   \texttt{row$_{t}$[]} is the same as
\texttt{row$_{t-1}$[]} except at the position where the new P-position
was found.)

We're going to need to keep track of another thing about these arrays.
The variables $\ell_{t}$ and $u_{t}$ are functions of the
\texttt{diag$_{t}$[]}.  Specifically $\ell_{t}$ is the minimim row
$y$ such that \texttt{diag$_{t}$[$y$]} is TRUE, and $u_{t}$ is the
maximum row $y$ such that \texttt{diag$_{t}$[$y$]} is TRUE.

\begin{figure}[hbt!]
  \centering
  \includegraphics[width=3in]{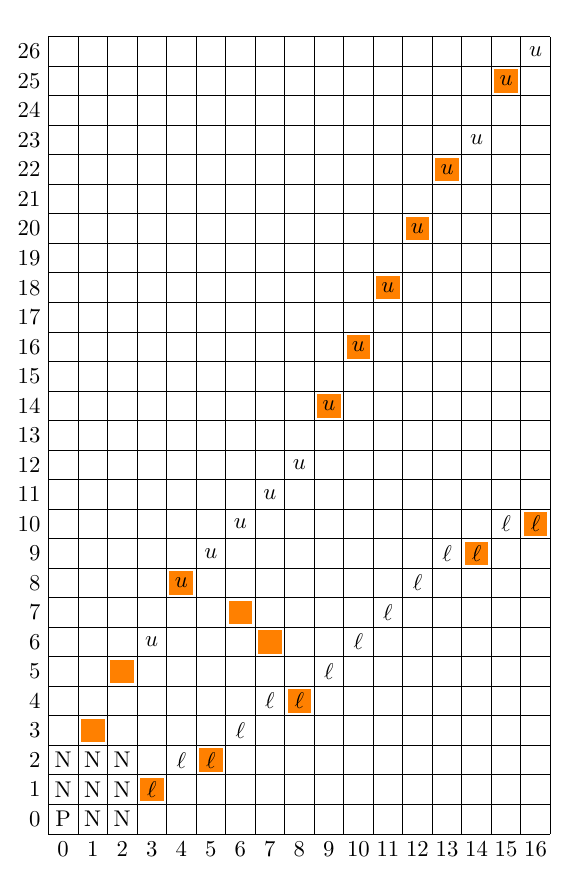}
  \captionsetup{width=.9\linewidth}
  \caption{This figure illustrates the computation of the P-positions
    and the definitions of $\ell_t$ and $u_t$, and the concept of
    saturation.  The cells labeled N and P are the elements of
    $\mathcal{P}$ and $\mathcal{N}$.  (This is the altered game that
    is shown in Figure~\ref{fig:two-game-diff}.)  The cells marked
    with an orange square are the computed P-positions.  In the
    columns from $t=3$ to $t=16$ the values of $\ell$ and $u$ are
    shown.  Beginning at column 3, all the columns are natural.  Only
    in columns 6 and 7 are the computed P-positions strictly between
    $\ell$ and $u$.  In all other cases the computed P-position is
    either $\ell$ or $u$.  The P-position in column 7 causes (for all
    subsequent columns) the \texttt{diag[]} array to be saturated,
    meaning that all of the diagonals between $\ell$ and $u$ contain a
    P-position.  }
  \label{fig:saturation}
\end{figure}

\section{Convergence Toward the Natural Wythoff Game}

\begin{lemma}
  \label{lemma:row-fill-property}  
  After a natural column $t$ is computed we know that for
  all $0 \leq y < \ell_t$ that \texttt{row}$_{\,t}[y]$ is TRUE.
  And we know that for all $y > u_t$ that \texttt{row}$_{\,t}[y]$ is FALSE.
\end{lemma}

\begin{proof}
  Consider the first part -- for $0 \leq y < \ell_t$ that we must have
  \texttt{row}$_{\,t}[y] =$ TRUE.  If $\ell_t = 0$ the statement is
  vacuously true.  So suppose $\ell_t \geq 1$.  Let $(t,y)$ be the
  P-position picked in column $t$.  We know $\ell_t \leq y$, because
  the diagonal through $(t,y)$ is used, and $\ell_t$ represents the
  lowest diagonal used.  When searching for the P-position the
  algorithm rejected $(t,0), (t,1), \ldots, (t,\ell_t-1)$.  None of
  those are on a used diagonal.  The only explanation for why they
  were rejected is because those rows are used.

  For the second part, for the sake of contradiction, suppose there is
  a row $y$ with $y > u_t$ and \texttt{row}$_{\,t}[y] =$ TRUE.  Let
  $t'\leq t$ be the column containing that P-position,  i.e. $(t',y)$
  is a P-position.  It immediately follows that the diagonal through
  $(t,y+t-t')$ is used.  But $y+t-t' \geq y > u_t$, contradicting the
  requirement that $u_t$ represent the highest used diagonal.
\end{proof}

A column $t$ is said to be \textit{saturated} if
\texttt{diag$_{\,t}$[$y$]} $=$ TRUE iff $\ell_t \leq y \leq u_t$.  In
this situation the values of $\ell_t$ and $u_t$ completely
characterize the set of diagonals that are used.

\begin{lemma}[Saturation Lemma]
  \label{lemma:saturation}
  There exists a natural column $t$ such that $t$ and all subsequent
  columns are saturated.
\end{lemma}

\begin{proof}
  For a column $t$, let $D_t$ be the number of unused diagonals ending
  at $(t,y)$, where $\ell_t \leq y \leq u_t$.  The column is not
  saturated iff $D_t>0$.
    
  Let $t$ be a natural column for which we're about to compute its
  P-position.  Suppose that $t-1$ is unsaturated.  There are three cases to
  consider based on where the P-position in column $t$ can reside.

  \begin{description}
  \item[Case 1:] The P-position is at $y = \ell_{t-1}$.  In this case
    $(\ell_t, u_t) = (\ell_{t-1}, u_{t-1} + 1)$.
  \item[Case 2:] The P-position is at $y = u_{t-1} + 2$.  In this case
    $(\ell_t, u_t) = (\ell_{t-1}+1, u_{t-1} + 2)$.
  \item[Case 3:] The P-position is at $y$ where $\ell_{t-1}+1 < y < u_{t-1}+1$.  In
    this case $(\ell_t, u_t) = (\ell_{t-1}+1, u_{t-1} + 1)$.
  \end{description}

  Because of Lemma~\ref{lemma:row-fill-property} these are the only
  three possibilities.  (See Figure~\ref{fig:saturation}.)

  In Cases 1 or 2 the number of unused diagonals does not change,
  i.e. $D_t = D_{t-1}$.  In Case 3 we have $D_t = D_{t-1}-1$.  We will
  show that Case 3 must happen after a finite number of steps.  It
  follows by induction that at some time $t'$ in the future that
  $D_{t'} = 0$.  This will prove the lemma.

  So our remaining task is to prove that Case 3 must happen
  eventually.  So let's assume the contrary and try to imagine a
  scenario where only Cases 1 and 2 occur.

  \begin{description}
  \item[Observation 1:] Case 1 cannot happen twice in a row.  Because
    when it happens row $\ell_t$ is used and $\ell_{t-1}=\ell_t$.  If
    it happened twice in a row that row would be used twice.

  \item[Observation 2:] When Case 2 happens it creates two unused
    neighboring rows in the range $[\ell_t, u_t]$, because in this
    case $u$ increases by two, introducing the two unused rows in that
    range.

  \item[Observation 3:] Any diagonal that is strictly between $\ell_t$
    and $u_t$ will always remain strictly between them as $t$ increases.
  \end{description}

  Let's just say for a second that there is exacly one unused diagonal
  between $\ell_t$ and $u_t$.  And using observation 2, we have two
  unused rows between $\ell_t$ and $u_t$.  We can now compute where
  these meet.  There will a point $(x,y)$ where the lower of the two
  unused rows meets the diagonal.  And another point $(x+1,y+1)$ where
  the upper of the two unused rows meets the diagonal.

  As we run forward computing column after column, and only Case 1 and
  Case 2 occur, eventually we get to a time $t=x$.  The point
  $(t,y)$ is on an unused diagonal between $\ell_t$ and $u_t$.  Also,
  row $y$ is not used.  So the only thing preventing selection of
  $(t,y)$ as the next P-position is the possibility that Case 1
  occurs.  Okay, so say that happens. Now for the next column $t+1$
  Case 1 cannot happen (Observaton 1).  So the search for a P-position
  continues up and before it reaches $(t+1,u_t+2)$ it will find
  $(t+1,y+1)$.  Selecting this P-position fills in the previously
  unused diagonal.

  Now we can remove the assumption that there is just one unused
  diagonal.  Because the only thing that can go wrong with the proof
  above is that the P-position covers some other unused diagonal and
  fills it in.  But this still reduces $D_t$ and makes progress
  towards the goal.

  This completes the proof.
\end{proof}

\noindent \textbf{Definition:} For any altered \wythoff{} game at time $t$
let $S_t$ be the binary string defined by starting with
$[\mbox{\texttt{row}}_t[\ell_t], \ldots, \mbox{\texttt{row}}_t[u_t]]$
and replacing each FALSE by $0$ and each TRUE by $1$.
\vspace{1em}

\noindent \textbf{Definition:} The \textit{Wythoff Update} of a non-empty binary
string $S$ is obtained from $S$ as as follows:
  \begin{quote}
    If the first character of $S$ is $0$ then change that character to
    a $1$ and append $0$ to the right end of $S$.  Alternatively, if the
    first character of $S$ is $1$ then remove it from $S$, and append
    $01$ to the right end of $S$.
  \end{quote}
Note that the update increases the length of the string by one.


\begin{figure}[hbt!]
  \centering
  \includegraphics[width=1.5in]{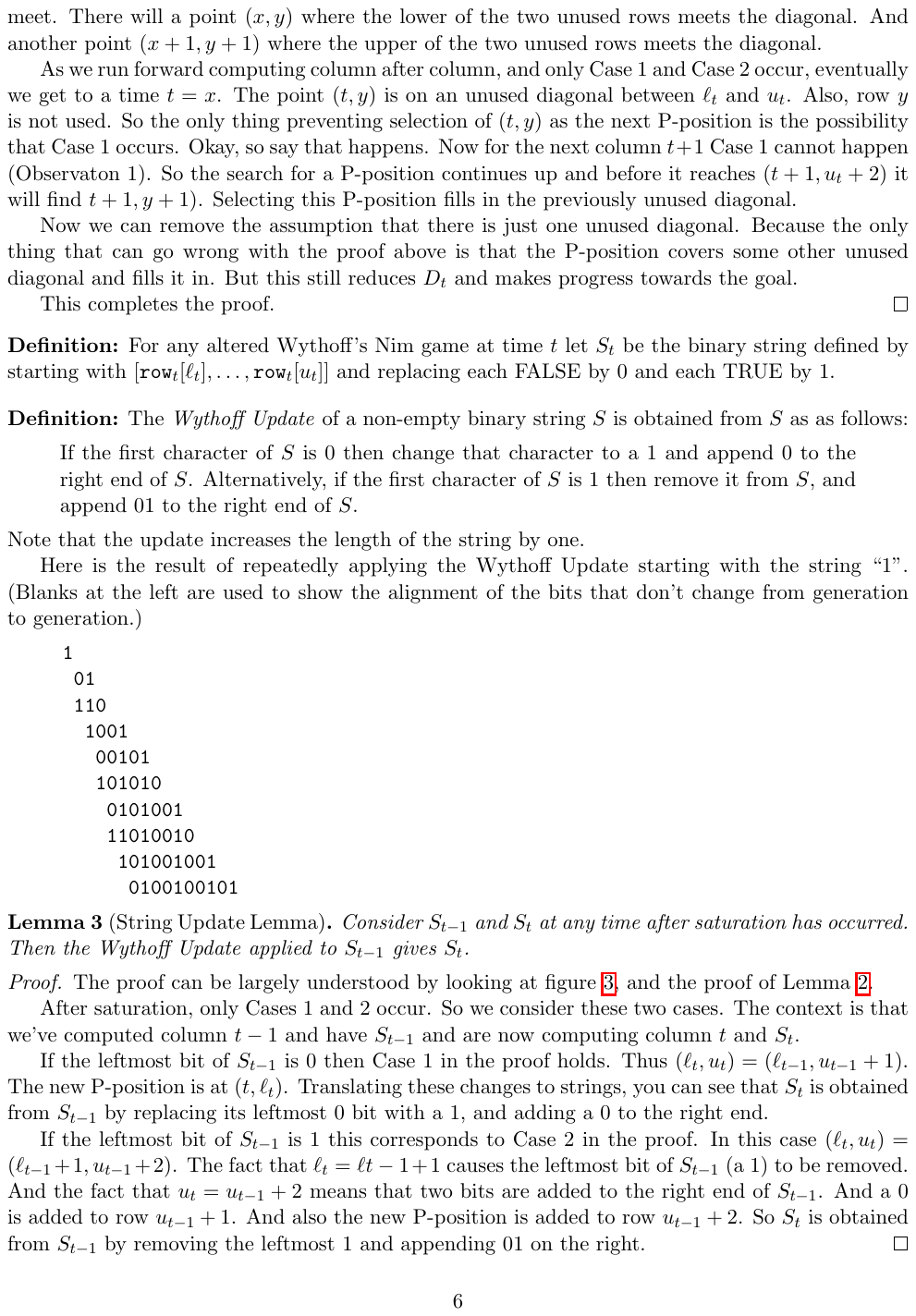}
  \captionsetup{width=.9\linewidth}  
  \caption{Here is the result of repeatedly applying the Wythoff Update starting
    with the string ``1''.  Alignment is used to emphasize the fact
    that most of the bits do not change from one iteration to the next.
  }
  \label{fig:wythoff-update}
\end{figure}

\begin{lemma}[String Update Lemma]
  \label{lemma:string-update}
  Consider $S_{t-1}$ and $S_t$ at any time after saturation has
  occurred.  Then the Wythoff Update applied to $S_{t-1}$ gives $S_t$.
\end{lemma}

\begin{proof}
The proof can be largely understood by looking at
Figure~\ref{fig:saturation}, and the proof of
Lemma~\ref{lemma:saturation}.

After saturation, only Cases 1 and 2 occur.  So we consider these two
cases.  The context is that we've computed column $t-1$ and have
$S_{t-1}$ and are now computing column $t$ and $S_t$.

If the leftmost bit of $S_{t-1}$ is $0$ then Case 1 in the proof
holds.  Thus $(\ell_t,u_t) = (\ell_{t-1}, u_{t-1}+1)$.  The new
P-position is at $(t,\ell_t)$.  Translating these changes to strings,
you can see that $S_t$ is obtained from $S_{t-1}$ by replacing its
leftmost 0 bit with a 1, and adding a 0 to the right end.
(This occurs when computing column $t=14$ in Figure~\ref{fig:saturation}.)

If the leftmost bit of $S_{t-1}$ is $1$ this corresponds to Case 2 in
the proof.  In this case $(\ell_t,u_t) = (\ell_{t-1}+1, u_{t-1}+2)$.
The fact that $\ell_t=\ell_{t-1}+1$ causes the leftmost bit of
$S_{t-1}$ (a 1) to be removed.  And the fact that $u_t = u_{t-1}+2$
means that two bits are added to the right end of $S_{t-1}$.  And a 0
is added to row $u_{t-1}+1$.  And also the new P-position is added to
row $u_{t-1}+2$.  So $S_t$ is obtained from $S_{t-1}$ by removing the
leftmost 1 and appending 01 on the right.  (This occurs when computing
column $t=15$ in Figure~\ref{fig:saturation}.)
\end{proof}

%
%

\section{Properties of the Wythoff String Update Rule}

\noindent \textbf{Definition:} Let $S$ and $T$ be bitstrings. We say that $S$ and $T$ are \textit{balanced} if they have the same length and the same number of ones.

\begin{lemma} 
Starting with two arbitrary bitstrings $S$ and $T$ of the same length, Wythoff updates are applied to both of them producing two sequences of strings $S=S_0, S_1, S_2, \ldots$ and $T=T_0, T_1, T_2, \ldots$. Then there exists a time $t$ when $S_t$ and $T_t$ become balanced.
\end{lemma}
\begin{proof}
Since the Wythoff update increases the length of a string by one, it follows that $|S_i| = |T_i|$ for all $i$.  For convenience, let $0_T$ and $1_T$ denote the number of zeros and ones in $T$ before applying any updates. Define $0_S$ and $1_S$ similarly.

If $0_S = 0_T$, the lemma is trivially true. Otherwise, without loss of generality, assume that $1_T > 1_S$. Since both strings have the same length, this implies that $0_T < 0_S$. Let $\delta = 1_T - 1_S$. We can induct on the value of $\delta$. Assume that for all $0 \leq \delta' < \delta$, the lemma is true.

By applying the Wythoff update to a string, the number of ones in the string increases by $1$ if the first bit is a $0$, and stays the same otherwise. So after applying $1_T + 2\cdot 0_T$ Wythoff updates, the number of ones in $T$ will increase by exactly $0_T$. Furthermore, since $0_S > 0_T$ and $1_S < 1_T$, there must exist some prefix $P$ of $S$ with exactly $0_T+1$ zeros and strictly less than $1_T$ ones. Let the number of ones in $P$ be $y$.  (Note that the last character of $P$ must be a 0.)

Consider the situation after $2\cdot 0_T + y + 1$ updates of $S$.  (This is one update for each 1 in $P$, and two updates for each 0 in $P$, except the last one, for which only one update is done.)  The result of these updates is that the number of ones in $S$ will have increased by exactly $0_T+1$. And $2\cdot 0_T + y + 1 \leq 2\cdot 0_T + 1_T$. Since the number of ones never decreases after a Wythoff update, the number of ones in $S$ after applying $2\cdot 0_T + 1_T$ updates is at least $0_T + 1$, and the value of $\delta$ must strictly decrease after a finite number of updates.

If $\delta$ remains positive after these updates, then we can use the induction hypothesis. Otherwise, $\delta$ becomes negative. The number of ones in a string increases by at most $1$ with each Wythoff update. So $\delta$ can change by an absolute value of at most $1$ with each Wythoff update, and there must have been some point in time when $S$ and $T$ had the same length and the same number of ones.
\end{proof}

\noindent \textbf{Definition:} Let $w$ be a function on bitstrings, which simultaneously replaces each \texttt{0} with \texttt{001}, and each \texttt{1} with \texttt{01}. For example, $w(\texttt{0010}) = \texttt{00100101001}$. For repeated application of $w$, we write $w^n$ --- for example, $w^5(S) = w(w(w(w(w(S)))))$. Note that applying $w$ is equivalent to applying $2\cdot 0_S + 1_S$ Wythoff updates to a bitstring $S$.

\vspace{1em}

There are several useful facts about $w$ when it is applied to a bitstring $S$, which will be useful later. All of them, listed below, can be easily proved inductively. 
\begin{itemize}
\item If $S$ is nonempty, the last character of $w(S)$ is a \texttt{1}, and the first character is a \texttt{0}.
\item No two consecutive characters of $w(S)$ are both ones.
\item No three consecutive characters of $w(S)$ are all zeros.
\item There are $2\cdot 0_S + 1_S$ zeros and $0_S + 1_S$ ones in $w(S)$.
\item The length of $S$ is equal to the number of ones in $w(S)$, or algebraically: $1_{w(S)} = |S| = 0_S + 1_S$.
\item From the previous two facts, the number of zeros in $S$ satisfies $0_S = 0_{w(S)} - 1_{w(S)}$.
\item It follows that $1_S = 2\cdot 1_{w(S)} - 0_{w(S)}$.
\item For any prefix $S'$ of $w(S)$ not ending in a zero, there exists a prefix $\sigma$ of $S$ satisfying $w(\sigma) = S'$. In particular, if $S'$ is the empty prefix, then $\sigma$ is also the empty prefix.
\end{itemize}

At this point, it will be helpful to introduce visuals for intuition. We represent a bitstring as a curve on a 2-D grid. A \texttt{0} corresponds to moving to the right, and a \texttt{1} corresponds to moving up. For example, the curve corresponding to \texttt{111001} would be drawn:

\begin{center}
    \includegraphics[scale=0.3]{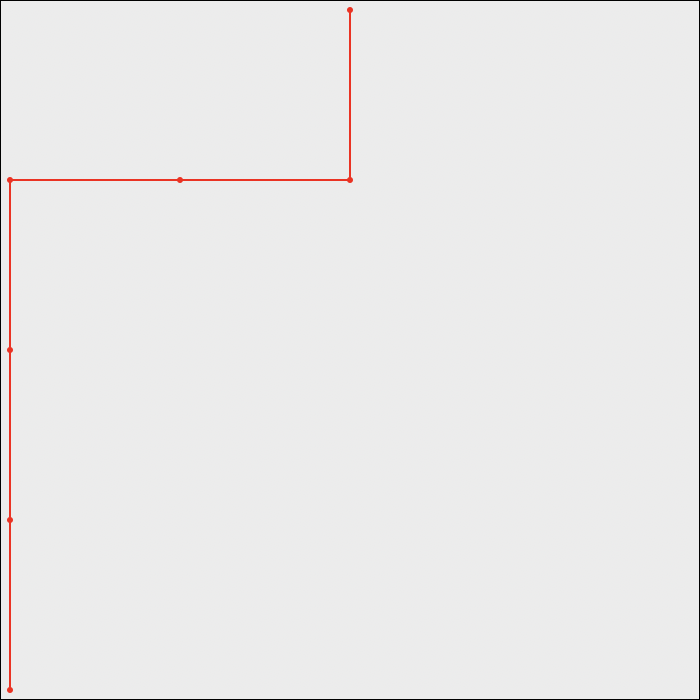}
\end{center}

By experimentation with small examples, we can observe that the area between two balanced strings appears to stay invariant under applications of $w$. For example, the evolution of \texttt{00100111} (in blue) and \texttt{11001100} (in red) under two applications of $w$ are pictured below. The area between the curves remains equal to $10$ throughout.

\begin{center}
    \includegraphics[scale=0.3]{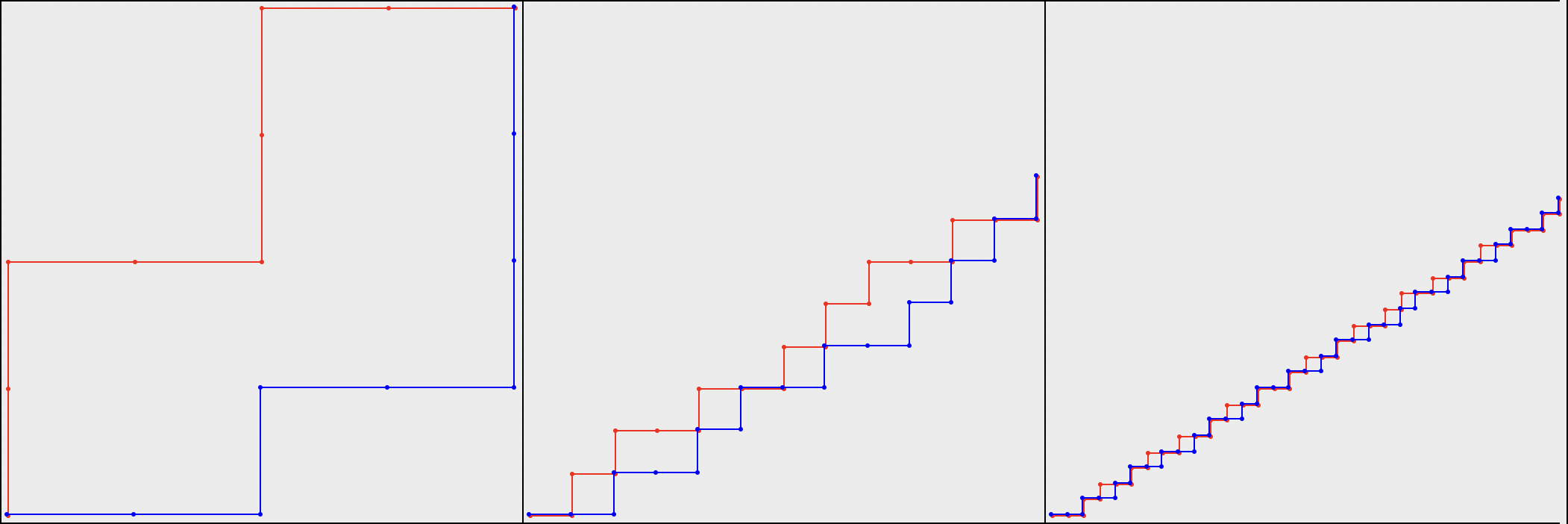}
\end{center}

This motivates the following definitions, which identify each unit of area with a coordinate:

\vspace{1em}

\noindent \textbf{Definition:} A unit of area is described by $(x, y)$, an ordered pair of integers. $(x,y)$ is said to be \textit{below} a string $S$ if there exists a prefix of $S$ with at most $x$ zeros and at least $y+1$ ones. Similarly, $(x,y)$ is \textit{above} $S$ if there exists a prefix of $S$ with at most $y$ ones and at least $x+1$ zeros.

\noindent \textbf{Definition:} A unit of area $(x, y)$ is \textit{between} two balanced strings $S$ and $T$, if it is above one of the strings and below the other. 

\noindent \textbf{Definition:} For balanced bitstrings, the \textit{area between $S$ and $T$} is the number of units of area $(x,y)$ between $S$ and $T$.

\vspace{1em}

The area between $S$ and $T$ must be finite, since for a unit of area to be below one string and above the other, it must satisfy $0 \leq x \leq 0_S$ and $0 \leq y \leq 1_S$. Also, a unit of area cannot simultaneously above and below a string $S$ ~--- since if $(x,y)$ is below $S$, then the shortest prefix with $x+1$ zeros must have at least $y+1$ ones, and so it cannot also be above $S$. 

\vspace{1em}

The following two lemmas provide insight into how units of area are transformed by applications of $w$:

\begin{lemma}
\label{forwardmap}
Let $S$ and $T$ be balanced bitstrings. If $(x,y)$ is between $S$ and $T$, then $(2x+y+1,x+y)$ is between $w(S)$ and $w(T)$.
\end{lemma}
\begin{proof}
Without loss of generality, suppose that $(x,y)$ is above $S$ and below $T$. First, we show that $(2x+y+1,x+y)$ is below $w(T)$:

\begin{itemize}
\item Let $T'$ be a prefix of $T$ with $x-l$ zeros and $y+k$ ones, for some $l \in \mathbb{N}$ and $k \in \mathbb{N}^+$. Such a prefix must exist by the definition of being below $T$.

$w(T')$ will have $2x-2l+y+k$ zeros and $x-l+y+k$ ones. Consider the shortest suffix of $w(T')$ with $k-l-1$ zeros. Since there are no two consecutive ones, and the since suffix must have a zero at its end, this suffix can have at most $k-l-1$ ones. So, removing this suffix from $w(T')$ results in a string with $2x+y+1-l$ zeros and at least $x+y+1$ ones. So $(2x+y+1, x+y)$ is below $w(T)$.
\end{itemize}

And a proof that $(2x+y+1,x+y)$ is above $w(S)$:

\begin{itemize}
\item The proof is similar to the one above. Let $S'$ be a prefix of $S$ with $y-l$ ones and $x+k$ zeros, for some $l \in \mathbb{N}$ and $k \in \mathbb{N}^+$. 

$w(S')$ will have $2x+y+2k-l$ zeros and $x+y+k-l$ ones. Consider the shortest suffix of $w(S')$ with $k-l$ ones. Since there are no three consecutive zeros, and since the suffix must have ones at both ends, the suffix can have at most $2(k-l-1)$ zeros. So, removing this suffix from $w(S')$ results in a string with at least $2x+y+l+2$ zeros and $x+y$ ones. So $(2x+y+1, x+y)$ is above $w(S)$.
\end{itemize}

So $(2x+y+1,x+y)$ is between $w(S)$ and $w(T)$.

\end{proof}

\begin{lemma}
\label{backwardmap}
Let $S$ and $T$ be balanced bitstrings. If $(x,y)$ is between $w(S)$ and $w(T)$, then $(x-y-1,2y-x+1)$ is between $S$ and $T$.
\end{lemma}
\begin{proof}
Without loss of generality, suppose that $(x,y)$ is above $w(S)$ and below $w(T)$.

Since it is above $w(S)$, there exists a prefix $S'$ of $w(S)$ with $y-l$ ones and $x+k$ zeros, for some $l\in\mathbb{N}$, and $k \in \mathbb{N}^+$. If the last character of $S'$ is not a one, extend its length until its last character is a one. This is possible because the last character of $w(S)$ must be a one. By performing this extension, we now have the inequalities $1_{S'} \leq y-l+1$ and $0_{S'} \geq x+k$. Now consider $w^{-1}(S')$ ~--- the prefix $\sigma$ of $S$ satisfying $w(\sigma) = S'$. 

The number of ones in $\sigma$ is 

\begin{align*}
1_\sigma &= 2\cdot 1_{S'} - 0_{S'} \\
&\leq 2\cdot(y-l+1) - (x+k) \\
&= 2y-x+1-(2l+k-1) \\
&\leq 2y-x+1
\end{align*}

And the number of zeros in $\sigma$ is 
\begin{align*}
0_\sigma &= 0_{S'}-1_{S'} \\
&\geq (x+k)-(y-l+1) \\
&= x-y-1+(l+k) \\
&\geq (x-y-1)+1
\end{align*}

So $\sigma$ is the prefix we are looking for to prove that $(x-y-1,2y-x+1)$ is above $S$. 

The proof that $(x-y-1,2y-x+1)$ is below $T$ proceeds similarly. Let $T'$ be a prefix of $w(T)$ with $x-l$ zeros and $y+k$ ones, for some $l\in\mathbb{N}$, and $k \in \mathbb{N}^+$. Reduce the length of $T'$ while it ends in a zero. Note that after this modification, the number of ones does not change. Since the number of ones is strictly positive, and $w(T)$ cannot start with a \texttt{1}, the number of zeros must also be positive. And we have $0_{T'} \leq x-l$ and $1_{T'} = y+k$. Let $\tau = w^{-1}(T')$. We can then derive:

\begin{align*}
1_\tau &= 2\cdot 1_{T'} - 0_{T'} \\
&\geq 2\cdot(y+k) - (x-l) \\
&= 2y-x+(2k+l) \\
&\geq (2y-x+1)+1
\end{align*}

And,

\begin{align*}
0_\tau &= 0_{T'}-1_{T'} \\
&\leq (x-l)-(y+k) \\
&= x-y-(y-k) \\
&\leq x-y-1
\end{align*}

So $(x-y-1,2y-x+1)$ is below $T$, and thus between $S$ and $T$.

\end{proof}

It's easy to verify that $f(x,y) = (2x+y+1,x+y)$ and $g(x,y) = (x-y-1,2y-x+1)$ are inverses of each other. It follows that $(x,y)$ is between $S$ and $T$ if and only if $(2x+y+1,x+y)$ is between $w(S)$ and $w(T)$.

\vspace{1em}

On closer inspection of the pictures, we can see that the area not only remains constant, but that the units of area get more and more separated with each application of $w$. 

\begin{center}
    \includegraphics[scale=0.3]{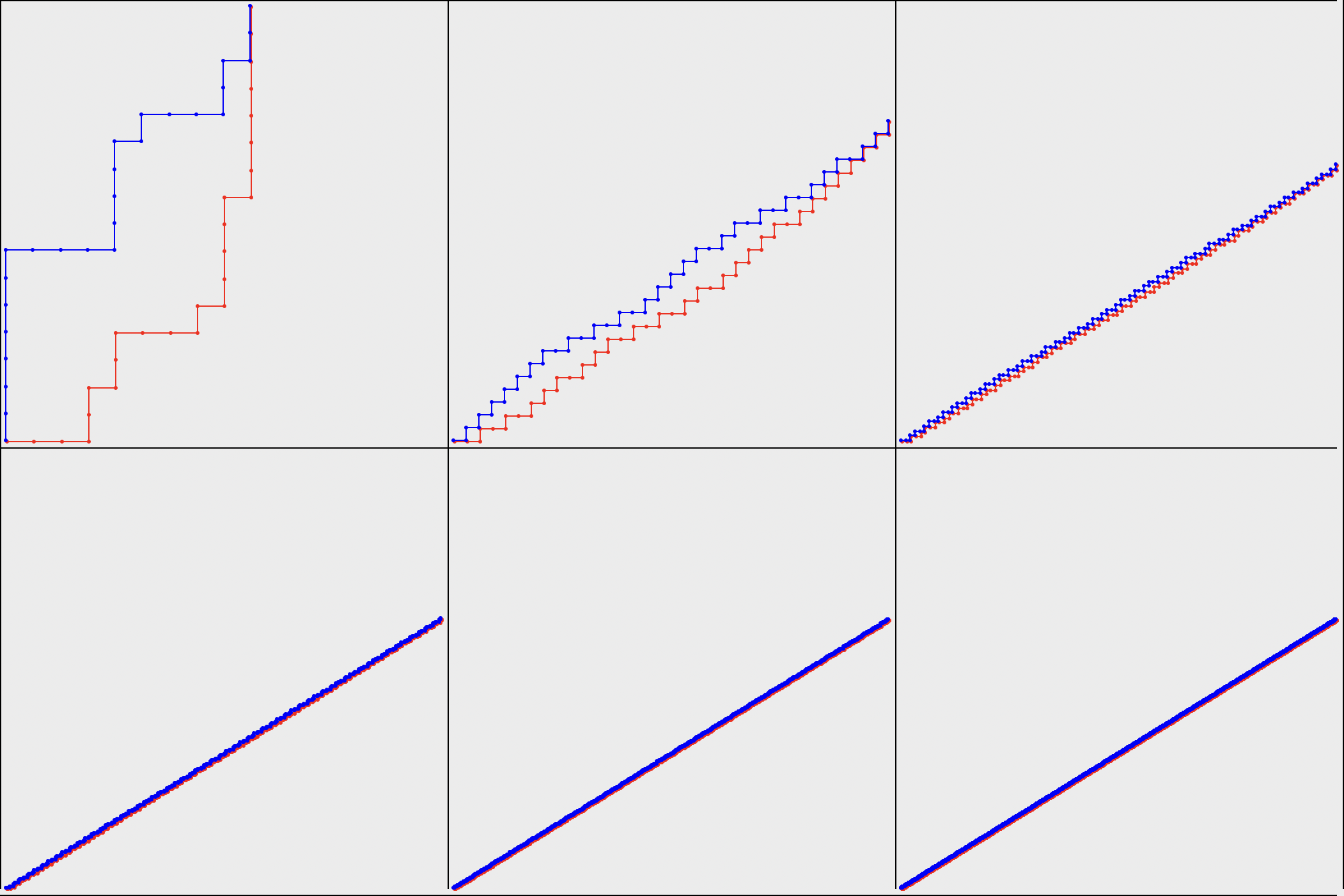}

    \small{A larger example, with more steps of evolution. The curves approach two lines,
    
    differing only by single disjoint units of area.}
\end{center}

We make this intuition precise with the notion of a \textit{defect}:

\vspace{1em}

\noindent \textbf{Definition:} For balanced bitstrings $S$ and $T$, a \textit{defect} is an index $1 \leq i \leq |S| - 1$ such that:
\begin{itemize}
\item The prefixes of $S$ and $T$ with length $i-1$ are balanced.
\item $S[i] \neq T[i]$
\item $S[i] \neq S[i + 1]$
\item $S[i + 1] \neq T[i + 1]$
\end{itemize}

\noindent \textbf{Definition:} Two bitstrings $S$ and $T$ \textit{differ by defects} if they are balanced, and for every index $1 \leq i \leq |S|$ where $S[i] \neq T[i]$, there is a defect either at index $i-1$ or $i$. In other words, every difference between $S$ and $T$ is a part of a defect.

\vspace{1em}

Visually, a defect is a single isolated unit of area between the two strings. And the strings differ by defects if no two units of area touch each other. We use this visual intuition to motivate the following lemmas:

\begin{lemma}
\label{diffdefect}
Let $S$ and $T$ be balanced. Then the following are equivalent: 
\begin{itemize}
\item For all units of area $(x,y)$ and $(x',y')$ between $S$ and $T$, $(x,y) \neq (x',y') \Rightarrow x\neq x'$ and $y\neq y'$.
\item $S$ and $T$ differ by defects.
\end{itemize} 

And, if they differ by defects, the number of defects is equal to the area between $S$ and $T$.

\end{lemma}
\begin{proof}

Suppose $S$ and $T$ satisfy the first condition, we prove that they differ by defects. Let $1 \leq v_1 < v_2 < \ldots < v_{\ell} \leq |S|$ be the indices where $S$ differs from $T$. The number of such indices must be even, otherwise the number of zeros in one string would be odd while the number of zeros in the other would be even, contradicting the definition of balanced.

If $\ell = 0$, $S$ and $T$ trivially differ by defects. Otherwise consider $v_1$ and $v_2$. The prefix of length $v_1-1$ is balanced, since that prefix is the same between the two strings. Let $z$ and $o$ be the number of zeros and ones in this prefix repsectively. 

If $S[v_1+1] = T[v_1+1] = \texttt{0}$, then both $(z, o)$ and $(z+1,o)$ would be between $S$ and $T$, contradicting the condition above. Similarly, if $S[v_1+1] = T[v_1+1] = \texttt{1}$, then both $(z, o)$ and $(z,o+1)$ would be between $S$ and $T$. 

So $v_2 = v_1+1$. And if $S[v_1] = S[v_2]$, then $(z,o)$, $(z+1,o)$, and $(z,o+1)$ would be between $S$ and $T$. So all the conditions for index $v_1$ to be a defect hold.

Finally, note that since $v_1$ is a defect, the prefixes of $S$ and $T$ with length $v_2$ must be balanced. So the prefix with length $v_3-1$ is balanced, and we can apply the same logic to indices $v_3$ and $v_4$, and so on, for each $v_{2i-1}$ and $v_{2i}$ ($1 \leq i \leq \frac{\ell}{2}$) to show that all indices where $S$ and $T$ differ are either defects, or come directly after a defect.

Now suppose $S$ and $T$ differ by defects. We will first show that there is a unit of area between $S$ and $T$ at $(x,y)$ if and only if there is a defect at index $x+y-1$. One direction is easy ~--- given a defect at index $i$, where there are $z$ zeros and $o$ ones in the prefix of length $i-1$, then $(z,o)$ will be between $S$ and $T$. 

Now consider any unit of area $(x,y)$ between $S$ and $T$. Without loss of generality, say it is above $S$ and below $T$. Since $S$ and $T$ differ by defects, any time they differ, causing those prefixes to become unbalanced, it gets corrected at the next index. So in any prefix $S'$ of $S$, and $T'$ of $T$ where $|S'| = |T'|$, the number of ones in $S'$ differs from the number of ones in $T'$ by at most $1$. So $(x,y)$ must correspond to a prefix of $S$ with $x$ ones and $y+1$ zeros, and a prefix of $T$ with $x+1$ ones and $y$ zeros. $S$ and $T$ only become unbalanced at their defects, so there must be a defect at index $x+y-1$.

Now consider the first defect at position $i$, let $z$ and $o$ be the number of zeros and ones respectively in the prefix of length $i-1$. There will be a unit of area at $(z,o)$. The prefixes of length $i$ are unbalanced, and become balanced again at $i+1$. So the earliest that the next defect can appear is at index $i+2$, which has strictly more zeros and ones. So the next unit of area must have a strictly greater $x$ and $y$ coordinate. And extending this reasoning, no two distinct units of area can share a $x$ or $y$ coordinate.

\end{proof}

\begin{lemma}
\label{fibmap}
Let $g(x,y) = (x-y-1,2y-x+1)$ as above. Let $F_i$ for $i \in \mathbb{Z}^+$ be the $i$-th Fibonacci number. Then for $n \in \mathbb{Z}^+$, $g^n(x,y) = (F_{2n-1}x - F_{2n}y-F_{2n}, -F_{2n}x+F_{2n+1}y+F_{2n+1}-1)$.
\end{lemma}
\begin{proof}

The proof is by induction, and follows easily from the statement of the lemma. We use $F_1 = 1, F_2 = 1, F_3 = 2, F_4 = 3, F_5 = 5, \ldots$.

\begin{itemize}
\item For $n=1$, we can verify that $g(x,y)=(x-y-1,2y-x+1) = (F_{1}x-F_{2}y-F_2,-F_2x+F_3y+F_3-1)$.
\item Let $n > 1$, and assume the lemma holds for all $n-1$. Then, we just need some dense but straightforward computation:

\begin{align*}
    g^n(x,y) &= g(g^{n-1}(x,y)) \\
    &= g(F_{2n-3}x - F_{2n-2}y-F_{2n-2}, -F_{2n-2}x+F_{2n-1}y+F_{2n-1}-1) \\
    &= ((F_{2n-3}x - F_{2n-2}y-F_{2n-2}) - (-F_{2n-2}x+F_{2n-1}y+F_{2n-1}-1) - 1, \\
    &\quad 2(-F_{2n-2}x+F_{2n-1}y+F_{2n-1}-1) - (F_{2n-3}x - F_{2n-2}y-F_{2n-2}) + 1) \\
    &= (F_{2n-3}x - F_{2n-2}y-F_{2n-2} + F_{2n-2}x-F_{2n-1}y-F_{2n-1}+1 - 1, \\
    &\quad -2F_{2n-2}x+2F_{2n-1}y+2F_{2n-1}-2 - F_{2n-3}x + F_{2n-2}y+F_{2n-2} + 1) \\
    &= ((F_{2n-3}+F_{2n-2})x - (F_{2n-2}+F_{2n-1})y - (F_{2n-2}+F_{2n-1}),\\
    &\quad -(2F_{2n-2}+F_{2n-3})x + (2F_{2n-1}+F_{2n-2})y + (2F_{2n-1}+F_{2n-2}-1)) \\
    &= (F_{2n-1}x-F_{2n}y-F_{2n}, -F_{2n}x+F_{2n+1}y+F_{2n+1}-1) \\
\end{align*}

\end{itemize}

\end{proof}

\begin{lemma}
\label{w-converge}
Let $S$ and $T$ be balanced. Then there exists an integer $n$, such that $w^{n}(S)$ and $w^{n}(T)$ differ by defects.
\end{lemma}
\begin{proof}
Let $n = |S|+1$. We will show that for any two distinct units of area $(x,y)$ and $(x',y')$ between $w^{n}(S)$ and $w^{n}(T)$, at least one of $x \neq x'$ and $y \neq y'$ is true.

Let $(x,y)$ and $(x',y')$ be between $w^n(S)$ and $w^n(T)$. Assume for sake of contradiction that $x = x'$, and let $\delta = y' - y$. From Lemma \ref{backwardmap}, $(x,y)$ and $(x',y')$ are between $w^n(S)$ and $w^n(T)$ iff $g^n(x,y)$ and $g^n(x',y')$ are between $S$ and $T$. And using the result from Lemma \ref{fibmap}, this means that the $y$-coordinates of $g^n(x,y)$ and $g^n(x',y')$ differ by

\begin{align*}
    (-F_{2n}x+F_{2n+1}y + F_{2n+1}-1) - (-F_{2n}x'+F_{2n+1}y'+F_{2n+1}-1) &= F_{2n+1}(y-y') \\
    &= F_{2n+1} \delta
\end{align*}

Since $\delta \neq 0$, this means that there must be two units of area between $S$ and $T$, whose $y$-coordinates differ by at least $F_{2n+1} \geq n = |S|+1$. But a unit of area can only be between $S$ and $T$ if its $y$-coordinate ranges from $0 \leq y \leq 1_S \leq |S|$. So these two units of area cannot exist, and no two distinct units of area between $w^n(S)$ and $w^n(T)$ have the same $x$-coordinate.

Similarly, we can assume for sake of contradiction that $y=y'$, and now let $\delta = x'-x$. Using the same argument, the corresponding units of area between $S$ and $T$ would have to have $y$-coordinates differing by

\begin{align*}
    (-F_{2n}x+F_{2n+1}y + F_{2n+1}-1) - (-F_{2n}x'+F_{2n+1}y'+F_{2n+1}-1) &= F_{2n}(x-x') \\
    &= F_{2n} \delta
\end{align*}

Meaning the $y$-coordinates would have to differ by at least $F_{2n} \geq n = |S|+1$, again a contradiction.

So $(x,y) \neq (x',y') \Rightarrow x\neq x' \text{ and } y \neq y'$, and we can apply Lemma \ref{diffdefect} to see that $w^n(S)$ and $w^n(T)$ differ by defects.

As a side note, $n=|S|+1$ is a very loose bound. Since the difference in coordinates grows on the order of the Fibonacci numbers, the actual number of applications of $w$ needed is asymptotically $\log_{\phi^2}(|S|)$. 

\end{proof}

\begin{corollary}[Convergence]
\label{corollary:convergence}
    Given any two bitstrings $S$ and $T$ of equal length, after applying a finite number of Wythoff updates, they will differ by defects.
\end{corollary}
\begin{proof}
    Apply Wythoff updates until the strings are balanced. Then repeatedly apply $w$ (which is equivalent to applying a finite number of Wythoff updates) as prescribed in Lemma \ref{w-converge} until the strings differ by defects.
\end{proof}

\begin{corollary}[Defect Preservation]
    Let $S$ and $T$ differ by defects. Then $w(S)$ and $w(T)$ also differ by the same number of defects.
\end{corollary}
\begin{proof}
    From Lemma \ref{diffdefect}, we know that for any two units of area $(x,y) \neq (x',y')$, either $x<x'$ and $y<y'$, or $x>x'$ and $y>y'$. Suppose that $x<x'$, then we can see that $2x+y+1 < 2x'+y'+1$ and $x+y<x'+y'$. So no two units of area in $w(S)$ and $w(T)$ share the same $x$ or $y$ coordinate, and we apply lemma \ref{diffdefect} again to see that they differ by defects.

    And the number of defects is equal to the area between $S$ and $T$. Since the area is preserved, the number of defects is preserved.
\end{proof}

\section{Asymptotic Equality and Offset}

We will now connect the results of the previous section about the
evolution of bitstrings under the Wythoff Update to the set of
P-positions of the two games.

Given a set $S$ of P-positions of a game, let 
$S_{(\Delta x, \Delta y)}$ be the same set offset by
$(\Delta x, \Delta y)$.  That is

$$
S_{(\Delta x, \Delta y)} = \{(x+\Delta x, y+\Delta y) \mid (x,y) \in S\}
$$

Let $S$ and $T$ be two infinite sets of points in $\mathbb{Z}^2$.  Let
$B_n = \{(x,y) \mid x^2+y^2\leq n^2\}$
We say that $S$ and $T$ are \textit{asymptotically equal} if
$$
\lim_{n \rightarrow \infty} \frac{|(S\cap T)\cap B_n|}{|(S\cup T)\cap B_n|} = 1
$$


It's not hard to see that asymptotic equivalence defines an equivalence
relation on subsets of $\mathbb{Z}^2$, which we will denote as $\sim$.

\begin{theorem}[Unique Offset Theorem]
  \label{theorem:unique-offset}
  Let $A$ be the set of P-positions of an altered \wythoff{} game.  Let
  $B$ be the set of P-positions of the natural \wythoff{} game.
  Then there exists a unique offset $(\Delta x, \Delta y)$ such
  that $A_{(\Delta x, \Delta y)}$ is asymptotically equal to $B$.
\end{theorem}

\begin{proof}
  By Lemma~\ref{lemma:saturation} we know that there is a column $t$
  in the altered game $A$ that is saturated, and all subsequent
  columns are saturated.  The length of the corresponding bitstring
  $S_t$ at time $t$ is $u_t - \ell_t + 1$, and this increases by one
  on each step.

  For the natural \wythoff{} game $B$ at time $t$ the length of the
  bitstring is $t+1$.  So $\Delta x$ is the amount of ``head start''
  we are going to give the natural game so that the two games align.
  So we have to define $\Delta x$ such that
  $$
  \Delta x + t + 1 = u_t - \ell_t + 1
  $$

  So $\Delta x = u_t - \ell_t - t$.  (Note that this is not a function
  of $t$ because both $u_t - \ell_t$ and $t$ increase by on one each time step.)

  As for the value of $\Delta y$,
  Corollary~\ref{corollary:convergence} tells us that there is a time $t$
  when the two games will have bitstrings that just differ by defects.
  Let $y_A = \ell_t$ in the altered game A.  We also need to look at
  what is happening in the natural game $B$ at time $t+\Delta x$.  It
  has a corresponding $y_B = \ell'_{t+\Delta x}$ (where $\ell'$
  indicates that it is taken from the natural game).  So we define
  $\Delta y = y_B - y_A$.  And shifting the altered game up by $\Delta
  y$ puts the altered game and the natural game into alignment.

  Now as we run both of the games forward, the only points in time
  when different P-positions are generated for the two games is when a
  defect is being processed.  A defect will cause one of the games to
  generate a P-position in the lower beam (Case 1 in the proof of
  Lemma~\ref{lemma:saturation}), and the other game to
  generate a P-position in the higher beam (Case 2 in the proof of
  Lemma~\ref{lemma:saturation}).  Then in the next step, the
  other bit of the defect is processed, and the roles are reversed.

  The bitstrings get longer and longer over time, but the number of
  defects remains the same.  Therefore the density of the defects in
  the bitstrings goes to zero.  It immediately follows that the two
  sets of P-positions (the translated altered game and the natural
  game) are asymptotically equal.

  We also need to show the uniqueness of the offset $(\Delta x, \Delta
  y)$.  Asymptotic equivalence is is an equivalence relation, and thus satisfies
  transitivity.  So if there are three sets $P$, $Q$, and $R$, then $P
  \sim Q$ and $Q \sim R$ implies $P \sim R$.
  


  Suppose that $A_{(\Delta x, \Delta y)} \sim B$, and also that
  $A_{(\Delta x', \Delta y')} \sim B$ for two distinct offsets.  It
  follows that $B_{(-\Delta x, -\Delta y)} \sim A \sim B_{(-\Delta
    x',-\Delta y')}$.  From which we infer by transitivity that:
  \begin{eqnarray*}
    & \\
    B_{(-\Delta x, -\Delta y)} & \sim & B_{(-\Delta x',-\Delta y')} \\
    B & \sim & B_{(\Delta x-\Delta x',\Delta y-\Delta y')} =
    B_{(\Delta x'', \Delta y'')}  \\
  \end{eqnarray*}
  where $(\Delta x'', \Delta y'')$ is a non-zero integer vector.  This
  implies we can keep moving the elements of the set $B$ by $(\Delta
  x'', \Delta y'')$, over and over again and continue to have
  asymptotically equality.  It's clearly impossible because all of the
  P-positions of $B$ lie very close two lines of different slopes.
  One of the lines must eventually have no overlap with its
  counterpart in the other set.  So the limit in the definition of
  asymptotic equality will have a value of at most $1/2$.
\end{proof}

\subsection{How to Compute the Offset}

After experimenting with various types of altered games, we were able to make
the following conjecture, which turned out to be true.  It will be proven below.
Let $\mbox{Box}(a,b)$ denote a box of width $a$ and height $b$ whose lower left
corner is $(0,0)$.

\begin{theorem}[Offset of a Box]
  \label{theorem:box-offset}
  Considered the altered game $A(\mathcal{P},\mathcal{N})$ where $\mathcal{N} =
  \{\}$ and $\mathcal{P} = \mbox{Box}(a,b)$.  Then $A_{(b-1,a-1)} \sim B$, where
  $B$ is the set of P-positions of the unaltered game of \wythoff{}.
\end{theorem}

\begin{figure}[hbt!]
  \centering
  \includegraphics[width=3in]{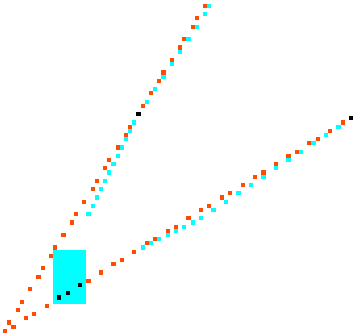}
  \captionsetup{width=.9\linewidth}
  \caption{This diagram illustrates Theorem~\ref{theorem:box-offset}.  The
    P-positions of \wythoff{} are shown in red.  The P-positions of the game
    altered by making all the cells in $\mbox{Box}(8,13)$ be P-positions
    (then offsetting it by a vector (12,7)) are shown in blue.  The P-positions within
    both games are shown in black.}
  \label{fig:n-box-8-13}
\end{figure}

Consider a general altered game $A(\mathcal{P}, \mathcal{N})$.  Let $t-1$ be
the column that is the last one containing any elements of the altered sets.
For this altered game, we can run the algorithm for computing the P-positions
described in section 2 of this paper.  We run it for each column from $0$ to
$t-1$.  After we're done with this process we can examine the situation.  We
compute three quantities from it.
\begin{eqnarray*}
  r & = & \mbox{The number of rows that have P-positions in them} \\
  d & = & \mbox{The number of diagonals that have P-positions in them} \\
  c & = & \mbox{The number of columns that have P-positions in them} \;\; = \;\; t
\end{eqnarray*}

With these definitions in mind, we can now state the following theorem.

\begin{theorem}[General Offset Theorem]
\label{theorem:general-offset}
Let $A$ be the P-positions of an altered game of
\wythoff{}, where $r$, $d$, and $c$ are computed as described above.  And let
$B$ be the P-positions of \wythoff{}.  Then
$A_{(d-c, d-r)} \sim B$.
\end{theorem}

\begin{proof}

For columns $t$ and beyond, one P-position is added.  And this P-position has
the property that it increases by one the number of diagonals covered, the number
of rows covered and the number of colums covered.  The same holds for the
natural game starting from column $0$.

If we start running the two games at the same time, after computing column $t-1$
the altered game will have $d$ diagonals, and the natural game will have $t$
diagonals.  We want these to be equal, but the number of excess diagonals in the
altered game is $d-t$.  Therefore we want to ``delay'' the start of the altered
game (by shifting it to the right) by $\Delta x = d-t = d-c$.  With this shift
the two games will eventually have exactly the same number of diagonals used.
(Note that this offset could be positive or negative.)

To compute the $\Delta y$ offset required we consider the number of rows covered
by the two games.  After computing column $t-1$ of the altered game we know that
it has used $r$ rows.  After computing the corresponding column for the natural
game (column $\Delta x + t-1$) we know that it has used $\Delta x + t$ rows.  But
$\Delta x + t = d-c+c = d$.  This means that going forward the altered game will
always have used $r-d$ more rows than the natural game.  (This of course could
be positive or negative.)

So let's skip ahead in time to a moment when the two games bitstrings differ by
defects.  We're going to consider the two values of $\ell$ from section 3 of the
paper.  Let $\ell$ be its value in the altered game at this time and let $\ell'$
be its value in the natural game at this time.  All the rows of each game below
their respective $\ell$s are used.  And as for the rows above $\ell$ and $\ell'$
respectively, there are the same number of rows used in both games.  Therefore
we know that $\ell-\ell' = r-d$.  Therefore $\ell'-\ell = d-r = \Delta y$, and
this is how much we have to shift up the altered game to align them.
\end{proof}

\begin{proof}[Proof Of Theorem~\ref{theorem:box-offset}]
    We can simply apply Theorem~\ref{theorem:general-offset}.  For
    $\mbox{Box}(a,b)$ we have the $c=a$, $r=b$, and $d=a+b-1$.  So the offset of
    the altered game making it equivalent to the natural game is $(d-c,d-r) =
    (a+b-1-a, a+b-1-b) = (b-1,a-1)$.
\end{proof}

\section{Final Remarks}

%
%
%

\textit{Linear Nimhoff} is a class of two-pile Nim games defined by
listing a set of \textit{rules}.  Each rule is described by a pair of
non-negative integers $(a,b)$.  Such a rule allows the player to move by
converting a pair of pile sizes $(x,y)$ to $(x-ka, y-kb)$, where $k$
is a positive integer, so long as the latter vector is non-negative.
Linear Nimhoff was defined by Friedman
\textit{et. al.}~\cite{geometric-analysis:2017}, extending the work of
Larsson~\cite{larsson2010generalizeddiagonalwythoffnim}.
Using this notation \wythoff{} is Linear Nimhoff with
this rule set: $\{(1,0),(0,1),(1,1)\}$.

It is natural to ask how other games in this class behave when
altered.  That is, will the phenomena that we observe in \wythoff{}
also occur in other Linear Nimhoff games?  We can summarize our
experimental results as follows:

\begin{enumerate}
\item We were unable to observe the phenomenon of defects in any other
  game.  That is, we looked for another game where the pattern of
  P-positions in the altered game gets closer and closer to the
  un-altered game as we move to larger and larger coordinates.  But we
  never found this in any game other than \wythoff{}.
\item On the other hand, the overall geometrical structure of the game
  does not seem to be changed when alterations are applied.  So for
  example if the original game has beams along roughly straight lines
  (these are the \textit{strict class} of games as defined
  in~\cite{geometric-analysis:2017}) then any finite alteration
  asymptotically preserves the same lines.
\end{enumerate}

To illustrate the first point, consider the Linear Nimhoff game
defined by this set of rules, which differs very slightly from those
of \wythoff{}: $\{(1,0), (0,2) (1,1)\}$.  We have
proven~\cite{mirabel-thesis} that the P-positions are comprised
of two linear beams defined by the sets
$\{(\floor{n \cdot \Delta x_1}, \floor{n\cdot \Delta y_1})\} \mid n\geq 0\}$
and $\{(\floor{n \cdot \Delta x_2}, \floor{n\cdot \Delta y_2})\} \mid
n\geq 0\}$. Where
\begin{eqnarray*}
  \Delta x_1 & = & \frac{1}{\sqrt{3}} \\
  \Delta y_1 & = & 1 + {\frac{1}{\sqrt{3}}} \\
  \Delta x_2 & = & 2 + \sqrt{3} \\
  \Delta y_2 & = & 1 + \sqrt{3} \\  
\end{eqnarray*}

\begin{figure}[hbt!]
  \centering
  \includegraphics[width=3in]{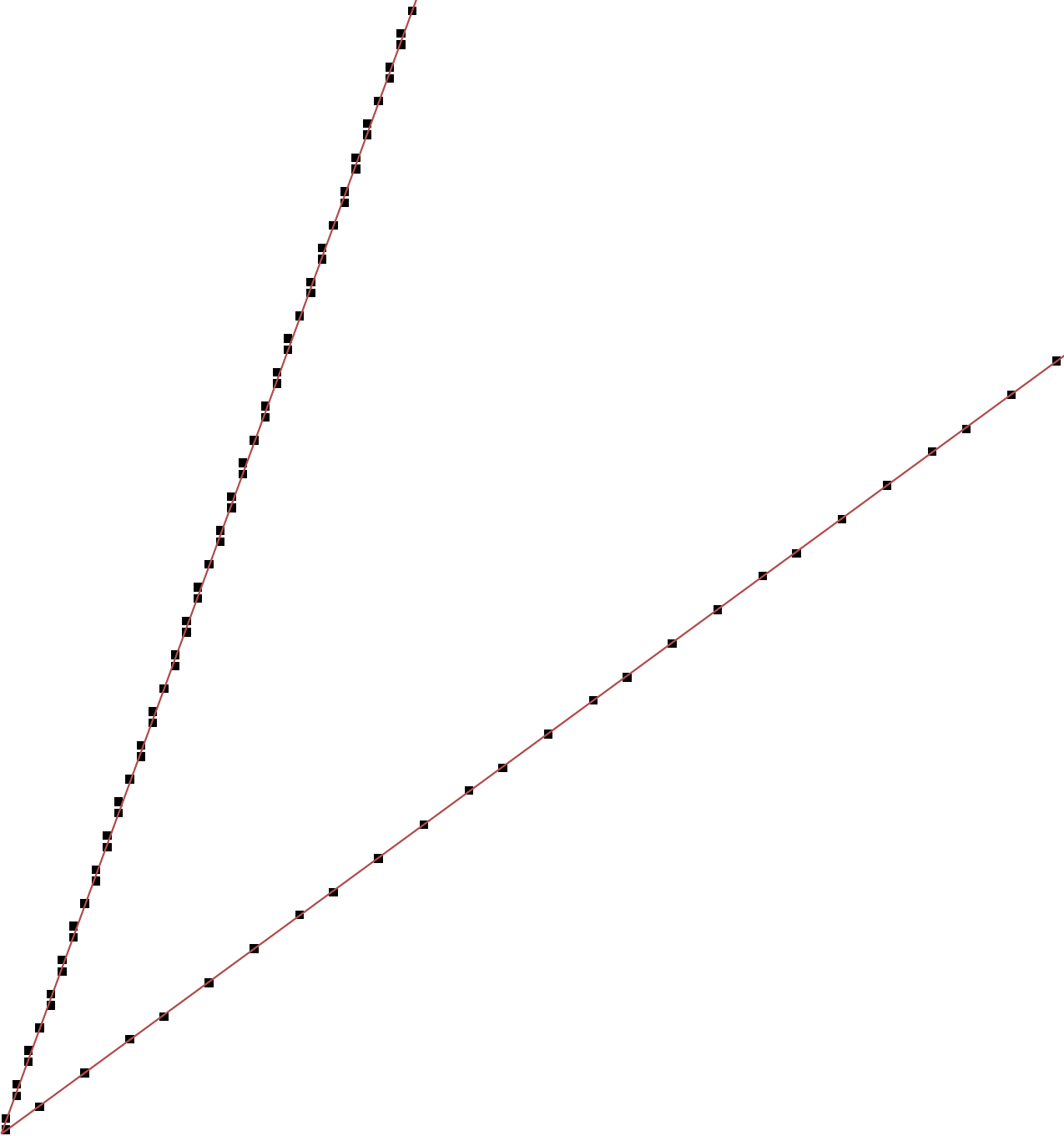}
  \includegraphics[width=3in]{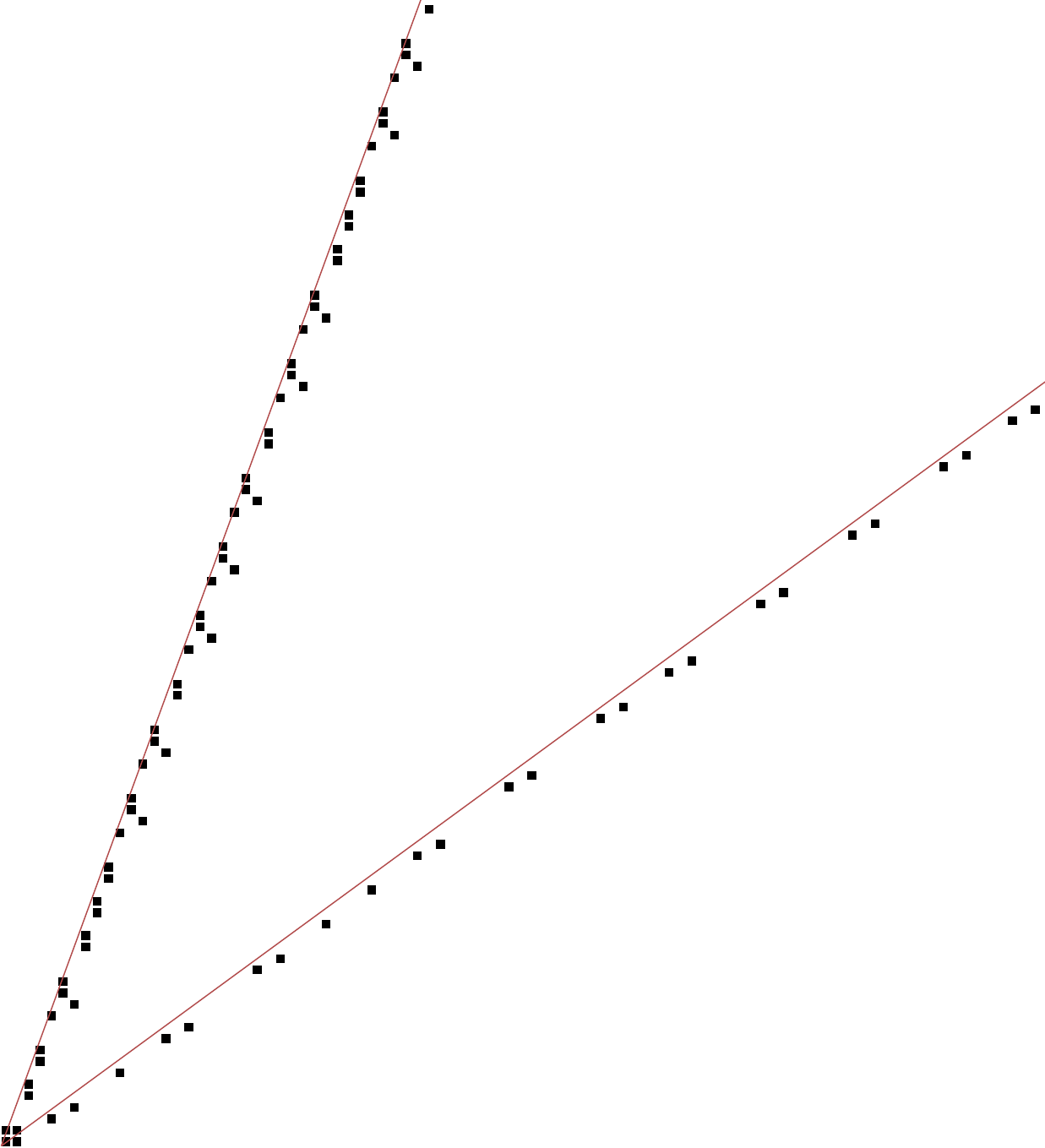}  
  \captionsetup{width=.9\linewidth}
  \caption{On the left is the Linear Nimhoff game with rules
    $\{(1,0), (0,2), (1,1)\}$.  On the right is an altered
    version of the same game with the lower left $2 \times 2$
    region labeled as P-positions. Notice that the differences between
    the two games (``defects'') do not peter out, but persist.}
  \label{fig:conc-nimhoff}
\end{figure}

Figure~\ref{fig:conc-nimhoff} shows the behavior of this game with a
small alteration.  The structure of the upper beam of the altered game
remains entirely different from that of the unaltered game, as far as
we've computed it.

\begin{figure}[hbt!]
  \centering
  \includegraphics[width=3in]{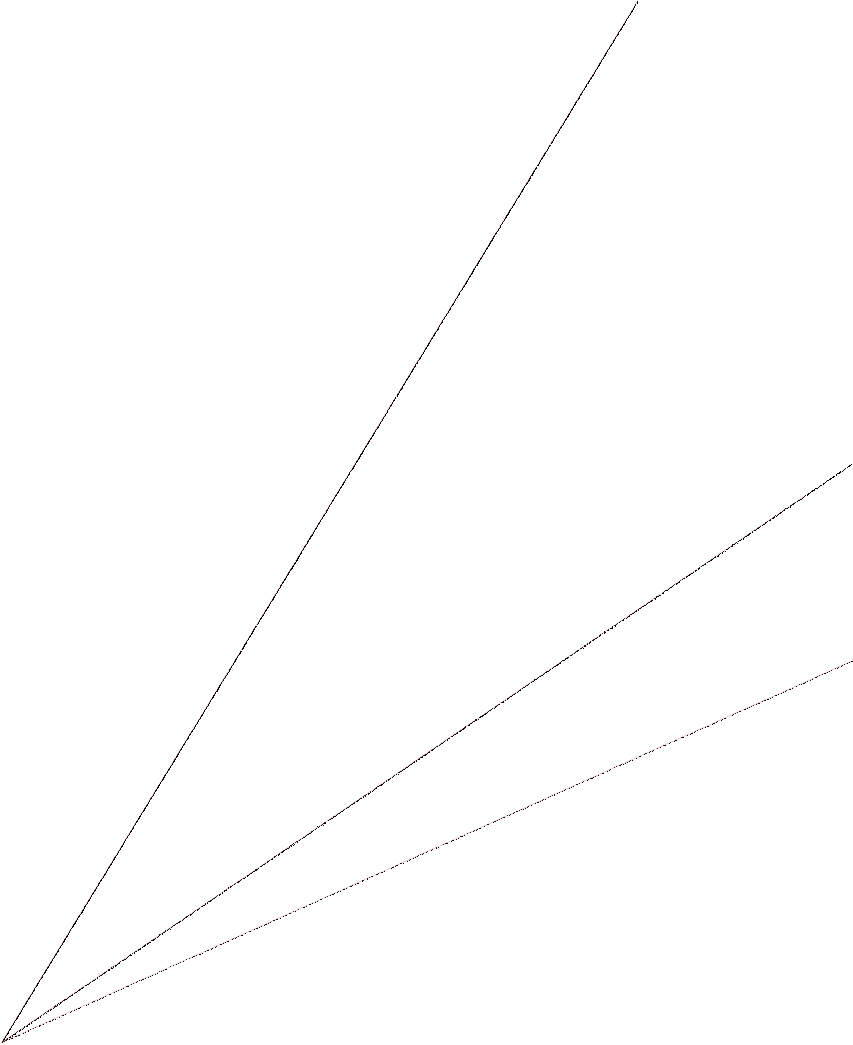}
  \includegraphics[width=3in]{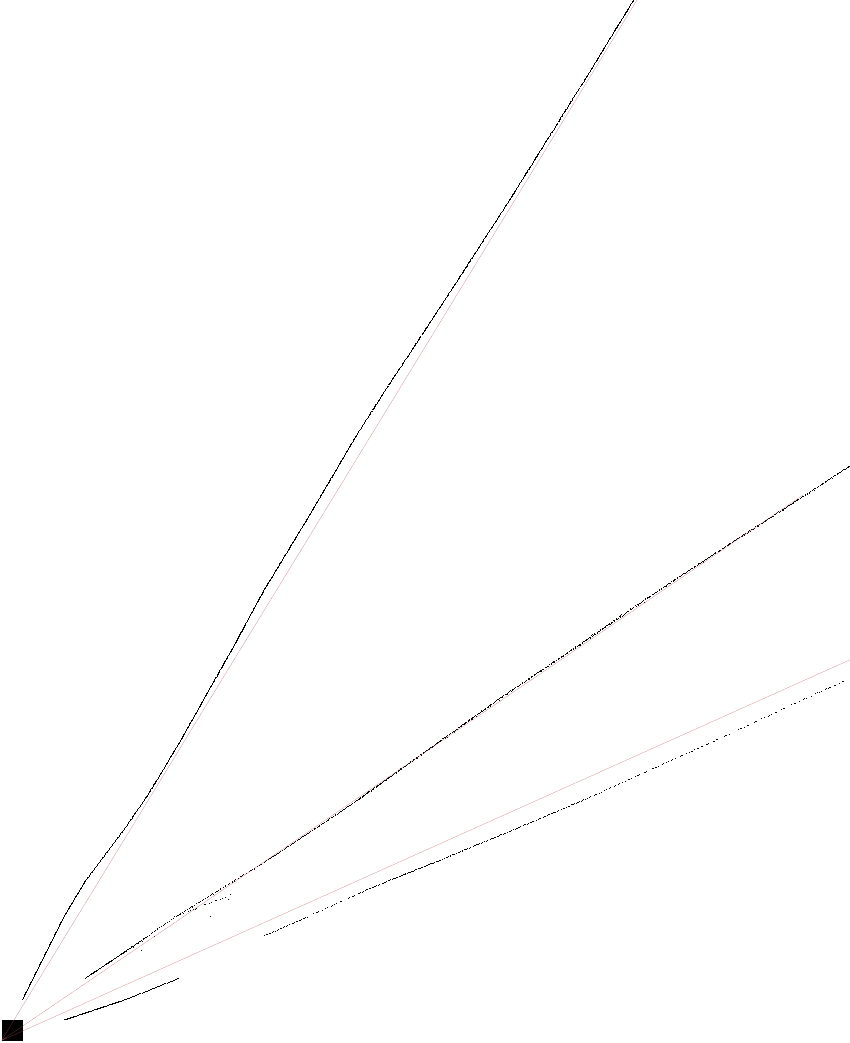}  
  \captionsetup{width=.9\linewidth}
  \caption{On the left is the Linear Nimhoff game with rules
    $\{(1,0), (0,1), (1,1), (2,1)\}$.  On the right is an altered
    version of the same game with the lower left $100\times 100$
    region labeled as P-positions.  The straight red lines drawn
    through the beams on the left side, are duplicated in the right
    side.}
  \label{fig:conclusion-nimhoff}
\end{figure}

The second point is illustrated in
Figure~\ref{fig:conclusion-nimhoff}.  Although initially different,
eventually the three beams of P-positions in the altered game conform
to the slopes of the corresponding beams in the un-altered game.

\clearpage

\bibliography{altered-wythoff}

\end{document}